\documentclass[12pt]{amsart}
\usepackage{amsmath,amssymb,verbatim,graphicx,hyperref}
\usepackage[cmtip,all]{xy}
\usepackage{amsfonts}
\usepackage{times}
\usepackage{amscd}
\usepackage{epsfig}
\usepackage{float}
\usepackage{amsthm}
\usepackage{latexsym}
\usepackage{imakeidx}
\usepackage[active]{srcltx}

\newtheorem{theo}{Theorem}[section]
\newtheorem{thm}[theo]{Theorem}
\newtheorem{lem}[theo]{Lemma}

\newtheorem{prop}[theo]{Proposition}

\theoremstyle{definition}
\newtheorem{dfn}[theo]{Definition}

\theoremstyle{remark}

\numberwithin{equation}{section}

\def\el{\mathrm{EL}}

\newcommand{\thu}{\mathrm{Th}}
\def\mod{\mathrm{mod}}

\def\Im{\mathrm{Im}}

\newcommand{\ol}{\overline}

\newcommand{\disk}{\mathbb{D}}

\newcommand{\la}{\lambda}
\newcommand{\ga}{\gamma}
\newcommand{\Ga}{\Gamma}

\newcommand{\ta}{\theta}

\newcommand{\om}{\omega}
\newcommand{\Om}{\Omega}
\newcommand{\al}{\alpha}

\newcommand{\eps}{\varepsilon}

\newcommand{\sm}{\setminus}

\newcommand{\uc}{\mathbb{S}}

\def\cuc{\mathcal{CU}}
\def\pal{\mathsf{p}_\alpha}

\def\A{\mathbb{A}}      \def\C{\mathbb{C}}

\def\R{\mathbb{R}}

\def\Z{\mathbb{Z}}

                \def\Cc{\mathcal{C}}
                \def\Fc{\mathcal{F}}

\newcommand{\Kc}{\mathcal{K}}

\def\Rc{\mathcal{R}}

\def\Zc{\mathcal{Z}}
\def\Pc{\mathcal{P}}

\def\Cc{\mathcal{C}}
\def\Uc{\mathcal{U}}       \def\Wc{\mathcal{W}}

\renewcommand\le{\leqslant}
\renewcommand\ge{\geqslant}
\def\0{\varnothing}
\def\d{\partial}

\begin{document}

\date{December 10, 2021}

\title{Upper bounds for the moduli of polynomial-like maps}

\author[A.~Blokh]{Alexander~Blokh}

\thanks{The second named author was partially
supported by NSF grant DMS--1807558}

\author[L.~Oversteegen]{Lex Oversteegen}

\author[V.~Timorin]{Vladlen~Timorin}

\thanks{The third named author was supported by the HSE University Basic Research Program.}

\address[Alexander~Blokh and Lex~Oversteegen]
{Department of Mathematics\\ University of Alabama at Birmingham\\
Birmingham, AL 35294}

\address[Vladlen~Timorin]
{Faculty of Mathematics\\
HSE University\\
6 Usacheva str., Moscow, Russia, 119048}

\email[Alexander~Blokh]{ablokh@uab.edu}
\email[Lex~Oversteegen]{overstee@uab.edu}
\email[Vladlen~Timorin]{vtimorin@hse.ru}



\begin{abstract}
We establish a version of the Pommerenke-Levin-Yoccoz inequality for the modulus of a
polynomial-like restriction of a global polynomial and give two applications.
First it is shown that if the modulus of a polynomial-like restriction of
an arbitrary polynomial is bounded from below then this forces bounded
combinatorics.
The second application concerns parameter slices of cubic polynomials given
by a non-repelling value of a fixed point multiplier. Namely, the
intersection of the main cubioid and
the multiplier slice lies in the closure
of the principal hyperbolic domain,
 with only possible exception of queer components.
\end{abstract}

\maketitle

\section{Introduction}
The 
Appendix (Section \ref{s:apx}) contains the necessary background.
We use standard notation ($\R, \C$, etc.).
The boundary (in $\C$) of a set $X\subset\C$ is denoted by $\d X$. For a set $Z$, let $|Z|$ be its cardinality.
For a polynomial $f$,
let $J_f$ be its Julia set, and $K_f$ be its filled Julia set.
Throughout the paper, $P$ denotes a polynomial of degree $D>1$ with connected Julia set $J_P$.

\subsection{Cuts and wedges}
\label{ss:cuts}
If external rays $R$ and $L$ of $P$ land at the same point $a$, then
 the union $\Gamma=R\cup L\cup\{a\}$ is called a \emph{cut}.
The point $a$ is called the \emph{vertex} of $\Gamma$.
The cut $\Gamma$ is \emph{degenerate} if $R=L$ and \emph{nondegenerate} otherwise.
Nondegenerate cuts separate $K_P$.
The \emph{period} of a periodic cut is the period of an external ray landing at its vertex.
A \emph{wedge} is a complementary component of a cut in $\C$.
We assume that cuts are oriented from $R$ to $L$ so that every cut $\Gamma$ bounds a unique wedge $W=W_\Gamma$
 where $\Gamma$ is the \emph{oriented} $\d W$ (i.e., if one walks along $\Gamma$ from $R$ to $L$, then
 $W_\Gamma$ is located on one's left side).
If $\Gamma$ is degenerate, then we set $W_\Gamma=\0$.
Say that a cut $\Gamma$ and the wedge $W_\Gamma$ are \emph{attached} to a set $T\subset K_P$
 if the vertex of $\Gamma$ belongs to $T$ but $W_\Gamma$ is disjoint from $T$.
A cut $\Gamma$ \emph{separates} a set $T_+\subset\C$ from a set $T_-\subset\C$ if
 $T_+\sm\Gamma\ne \0$ and $T_-\sm\Gamma\ne \0$ lie in different components of $\C\sm\Gamma$.

\subsection{The modulus of a PL map}
\label{ss:modren} From now on replace the expression ``polynomial-like'' by
``PL''; filled PL Julia sets are always denoted by  $K^*$, and $P$ always
means a polynomial of degree $D$.

\begin{dfn}[\cite{lyu97}]\label{d:mod-pl}
Given a PL restriction $P:U_1\to U_0$ of $P$, call $\mod(U_0\sm\ol U_1)$
\emph{the modulus of the PL map} $P:U_1\to U_0$.
\end{dfn}

Let us introduce the topological concept of a \emph{core component}.

\begin{dfn}\label{d:core}
Let $\Gamma=R\cup L\cup\{a\}$ be a cut with a repelling/parabolic
$P$-periodic vertex $a\in U$ where $U$ is an open Jordan disk. Choose a small
open Jordan disk $\Delta'\subset W_\Gamma\cap U$ whose boundary consists of
two small initial arcs of $L$ and $R$ with common endpoint $a$ and other
endpoints $x_L\in L,$ $x_R\in R$, and a curve $I$ connecting $x_L$ and $x_R$
inside $W_\Gamma\cap U$. Denote by $\Delta_{\Gamma, U}$ the component of
$W_\Gamma\cap U$ containing $\Delta'$ and call $\Delta_{\Gamma, U}$ the
\emph{core component} of $W_\Gamma\cap U$. The core component
$\Delta_{\Gamma, U}$ is independent of the choice of $\Delta'$.
\end{dfn}

To study the PL (filled) Julia sets, we need a few 
other concepts.

\begin{dfn}\label{d:paralegal}
For a PL restriction $P:U_1\to U_0$ of $P$,
 let $\Zc$ be a finite $P$-invariant set of periodic nondegenerate cuts
 attached to $K^*$ (whose vertices, called \emph{$\Zc$-vertices},
 are repelling or parabolic $P$-periodic points), and let
 $\Wc_\Zc$ be the set of the associated wedges (\emph{$\Zc$-wedges}).
If the wedges $W_{\Gamma}$ from $\Wc_\Zc$
are pairwise disjoint, and for each $\Gamma\in \Zc$ the restriction
$P:\Delta_{\Gamma, U_1}\to\C$ is univalent, $\Zc$ is called \emph{paralegal},
 see Fig. \ref{fig:legal}.
\end{dfn}

\begin{figure}
  \centering
  \includegraphics[width=7cm]{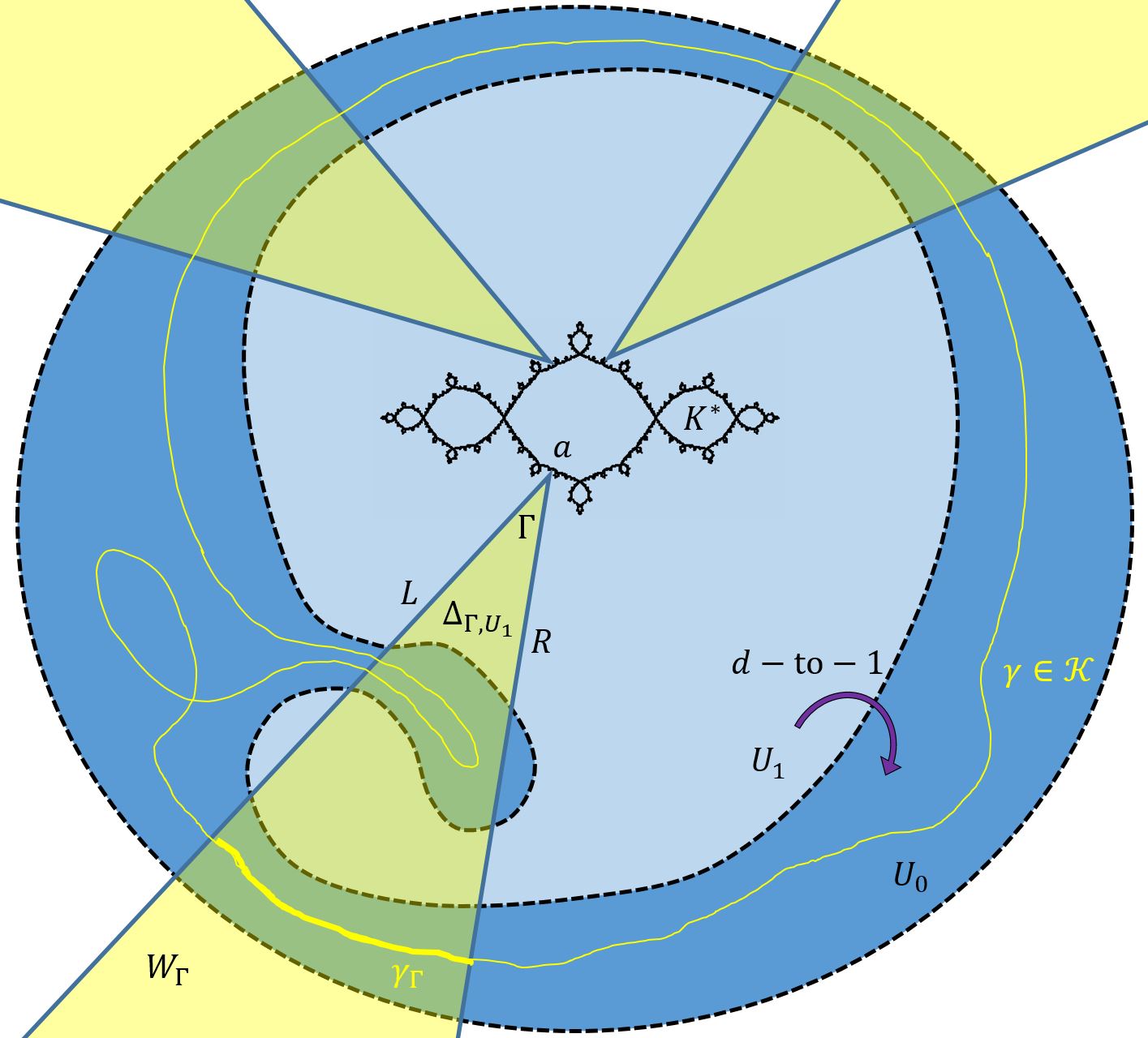}
  \caption{\small A paralegal set of cuts (Definition \ref{d:paralegal}, schematic illustration).
  The domains $U_0$, $U_1$ are shown as regions with dashed boundaries.
  For a cut $\Gamma\in\Zc$, the boundary rays $R$, $L$, and their landing point $a$ are shown,
  as well as the set $\Delta_{\Gamma,U_1}$ (on which $P$ is univalent by our assumptions).
  A curve $\gamma\in\Kc$ and its segment $\gamma_\Gamma\in\Kc_\Gamma$ are also shown
  (see Definition \ref{d:Kc}).}\label{fig:legal}
\end{figure}

Below $\Zc$ \emph{always} denotes a paralegal set of cuts of a PL restriction
$P:U_1\to U_0$ of $P$. A $\Zc$-wedge may contain external rays landing at its
vertex as the external rays forming cuts of $\Zc$ with vertex $z$ do not have
to be \emph{all} external rays landing at $z$.

\begin{thm}
  \label{t:nopar0}
Let $P:U_1\to U_0$ be a PL restriction of $P$.
Then
$$\frac 1{\mod(U_0\sm\ol U_1)}\ge  \frac{\pi
|\mathcal
Z|}{\log D}.$$
\end{thm}

For Theorem \ref{t:nopar}, a quantitative version of Theorem \ref{t:nopar0},
we need additional concepts.

\begin{dfn}\label{d:z-domain}
If $a$ is a vertex of a $\Zc$-wedge $W_\Gamma$ and a Fatou domain
$\Omega\subset W_\Gamma$ contains a periodic access $\alpha$ to $a$, then
$\Omega$ is called a \emph{$\Zc$-domain} (\emph{$W_\Ga$-domain}), $\alpha$ is
called a \emph{$\Zc$-access} (\emph{$W_\Ga$-access}), and the \emph{period}
of $\al$ is denoted by $m_\al$. 
\end{dfn}

Note that the boundary of $U_1$ cuts through $\Omega$ bypassing critical points in $\Omega$.
If $\Omega$ is a $\Zc$-domain, then $\Omega$ is not a Siegel domain.
The set $\Omega\cap K^*$ is empty.
Indeed, otherwise it is easy to see that all domains from the orbit of $\Om$ are contained in $K^*$.
Thus, all domains from the orbit of $\Om$ are contained in the core components of the corresponding wedges,
a contradiction with the fact that $P$ on these core components is univalent.
If $a$ is parabolic
(with corresponding parabolic Fatou domains $\Omega^*\subset K^*$ and $a\in \d\Omega^*$),
 the domain $\Omega\subset W_\Ga$ is not its immediate basin of attraction.

\begin{dfn}\label{d:legal}
Define three sets of periodic accesses to the vertices of $\Zc$:
\begin{itemize}
  \item the set $\mathfrak{B}_\Zc$ of all periodic accesses to $\Zc$-vertices from $\Zc$-domains,
  \item the set $\mathfrak{A}_\Zc$ of accesses from infinity to $\Zc$-vertices represented by $R$ for every
$\Gamma=R\cup L\cup\{a\}\in\Zc$,
  \item the set $\mathfrak{C}_\Zc$ of accesses from infinity to $\Zc$-vertices corresponding to
   external rays contained in $\Zc$-wedges.
\end{itemize}
The paralegal set $\Zc$ of cuts endowed with sets $\mathfrak{A}_\Zc$, $\mathfrak{B}_\Zc$, $\mathfrak{C}_\Zc$
 is said to be \emph{legal}.
Note, that, by definition, $|\mathfrak{A}_\Zc|=|\Zc|$.
\end{dfn}

Let $a$ be a vertex of $\Zc$ and $\alpha$ be a $\Zc$-access to $a$ from a
$\Zc$-domain $\Omega$.
A Riemann map $\phi:\disk\to\Omega$ depends on $\Omega$, not on $\alpha$.
By \cite{Pom86}, there is a unique point $b\in\uc$ such that $\phi^{-1}(\alpha)$ is an access to $b$ from $\disk$.
The Blaschke product $\phi^{-1}\circ P^{m_\alpha}\circ\phi$ has a
multiplier $\lambda^*_\alpha\in\R_{>1}$ at $b$ called the \emph{conjugate multiplier} (of $\alpha$).
Necessarily, $\lambda^*_\al>1$ as $\lambda^*_\al=1$ means that the
$\Zc$-domain $\Om_1$ with vertex $a$ is an immediate basin of the parabolic
point $a$, which is impossible.

\begin{thm}
  \label{t:nopar}
Let $P:U_1\to U_0$ be a PL restriction of $P$; let $\Zc$ be a legal set of cuts.
Then
$$
\frac 1{\mod(U_0\sm\ol U_1)}\ge  \sum_{\alpha\in\mathfrak{A}_\Zc\cup\mathfrak{B}_\Zc\cup\mathfrak{C}_\Zc}
\frac{m_\alpha\pi}{\log\la^*_\alpha}=
\frac{\pi(|\mathcal Z|+|\mathfrak{C}_\Zc|)}{\log D}+
\sum_{\alpha\in\mathfrak{B}_\Zc} \frac{m_\alpha\pi}{\log\la^*_\alpha}\ge
$$
$$
\ge\frac{\pi (|\mathcal Z|+|\mathfrak{C}_\Zc|)}{\log D}.
$$
\end{thm}

The assumption that the wedges $W_\Gamma$ contain no critical points
 in $\Delta_{U_1,\Gamma}$ is satisfied if, e.g., the filled Julia set of the PL map $P:U_1\to U_0$
  is connected and disjoint from all $W_\Gamma$.
The right hand side of the inequality is independent of the PL degree $d$ of $P:U_1\to U_0$.
The case $d=1$ is \emph{not} excluded, rather it is closely related with a special case of Theorem 3 from \cite{Pom86}
(see below).

To illustrate how these theorems work, we give two applications: the first is
dynamical and valid for any degree $D$, the second deals with 1-dimensional
parameter slices of the space of cubic polynomials.

\subsection{Further discussion}\label{ss:fd}
To relate Theorem \ref{t:nopar} to known results we use our 
 machinery and specialize Theorem 3 of \cite{Pom86} in the polynomial case.
For a PL map $g:U_1\to U_0$ of degree one with repelling fixed point $a\in
U_1$ of combinatorial rotation number $0$ take the cuts formed by \emph{all}
external rays landing at $a$ associated to minimal by inclusion wedges; with
sets of accesses $\mathfrak{B}_a$ and $\mathfrak{A}_a$ (see Definition
\ref{d:legal}) this defines a legal set of cuts $\Zc_a$. 

\begin{thm}
  [\protect{\cite[Theorem 3]{Pom86}}] \label{t:pom}
Consider a degree $D>1$ complex polynomial $g$ with connected Julia set. 
Let $a$ be a repelling $g$-fixed point of combinatorial rotation number $0$.
Then
$$
\frac{2}{\log|g'(a)|}\ge \frac{2\log|g'(a)|}{|\log g'(a)|^2}\ge
\sum_{\alpha\in\mathfrak{A}_a\cup\mathfrak{B}_a} \frac{1}{\log\la^*_\alpha}=\frac{|\mathfrak{A}_a|}{\log D}+
\sum_{\alpha\in\mathfrak{B}_a} \frac{1}{\log\la^*_\alpha}.
$$
\end{thm}

Let $g:U_1\to U_0$ be a degree one PL restriction of a polynomial $g$ with
$a\in U_1$. A straightforward computation shows that $2\pi\, \mod(U_0\sm\ol
U_1)\le\log |g'(a)|$; equality is attained if $U_1$ is represented by a round
disk in the linearizing coordinate for $f$ near $a$. Substituting this
expression into Theorem \ref{t:pom}, removing intermediate terms, and using
the fact that $|\mathfrak{A}_a|=|\Zc_a|$ and $m_\al=1$ for any access $\al\in
\mathfrak{B}_a$, we obtain the inequality
\[
\frac 1{\pi\, \mod(U_0\sm\ol U_1)}\ge \frac{|\Zc_a|}{\log D}+
\sum_{\alpha\in\mathfrak{B}_a} \frac{1}{\log\la^*_\alpha}
\]
which is precisely the case $d=1$ of Theorem \ref{t:nopar} for rotation number $0$.

Theorem 3 of \cite{Pom86} was later generalized by Levin \cite{Lev90,Lev91} and is now 
a part of the more general Pommerenke--Levin--Yoccoz (PLY) inequality \cite{Hub93, Pet93}.
All versions of the PLY inequality deal with a single fixed (or periodic) point $a$ of $P$.
The generalization of Theorem \ref{t:pom} by Levin \cite{Lev91} has two improvements.
Firstly, the Riemann maps $\phi$ are replaced with $\varkappa$-quasiconformal maps.
Then, in the right hand side of the inequality, the term $(\log\la^*_\alpha)^{-1}$ is replaced with
 $(\varkappa\log \lambda^*_\alpha)^{-1}$.
Secondly, the left hand side can be replaced with $2\beta/(\log |g'(a)|)$.
Here $\beta\in [0,1]$ is the asymptotic density of
$\bigcup_{\alpha\in\mathfrak{A}\cup\mathfrak{B}} \Omega_\alpha$
 near $a$ with respect to the metric $|dz|/|z-a|$, and $\Omega_\alpha$ is the Fatou component containing $\alpha$.
Similar improvements can also be made to Theorem \ref{t:nopar} with essentially the same methods.

The contribution of Yoccoz \cite{Hub93,Pet93} deals with nonzero rotation numbers.
Though the $d=1$ case of Theorem \ref{t:nopar} includes the possibility of a
nonzero rotation number,
 it is essentially reduced (via the summation trick, see Section \ref{ss:trick}) to the non-rotational case and,
 as a consequence, is weaker than the full PLY inequality.
Note that, for higher degrees, Theorem \ref{t:nopar} describes the influence of several different cycles,
 which is not the case for the PLY inequality.
Another interesting analog of the PLY inequality is obtained in \cite{BE02},
 however, it is not compatible with PL behavior.
There are versions of the PLY inequality in the case when the Julia set of $g$ is allowed to be disconnected,
 see \cite{EL93,Lev96}.

\subsection{Bounded geometry implies bounded combinatorics}\label{ss:bb}
Theorem \ref{t:bound-period} is a dynamical application of Theorem \ref{t:nopar}.

\begin{thm}
\label{t:bound-period} Let $P$ be a degree $D$ polynomial with connected
$K_P$. If for a PL restriction $P:U_1\to U_0$ of modulus $\mod(U_0\sm\ol
U_1)\ge \mu$ and filled PL Julia set $K^*$ there exists a cycle of cuts $\Zc$
of minimal period $s$ attached to $K^*$, 
then
$s\le\frac{\log D}{\mu\pi}.$ In particular, there are only finitely many
possible pairs of arguments of external rays that form $\Zc$.
\end{thm}

\begin{proof}
By Lemma \ref{l:sibling-inside}, the collection $\Zc$ is legal.
By Theorem \ref{t:nopar0},
we have $\mu\le\mod(U_0\sm\ol U_1)\le(s\pi)^{-1}\log D$.
It follows that $s\le\frac{\log D}{\mu\pi}$.
\end{proof}

\subsection{Slices of cubic polynomials}\label{ss:slices1}
Consider the space of complex cubic polynomials with fixed point $0$.
By a linear conjugacy (that is, a map $z\mapsto\varkappa z$ with $\varkappa\in\C\sm\{0\}$),
 any such polynomial can be reduced to the form
$$
f_{\la,b}(z)=\la z+bz^2+z^3.
$$
Let $\Fc$ be the space of all such polynomials, and $\Fc_\la$ be the space of $f_{\la,b}$ with fixed $\la$.
Then $\Fc_\la$ is isomorphic to $\C$, and $b\in\C$ is a natural complex coordinate on $\Fc_\la$.
The maps $f_{\la,\pm b}$ are linearly conjugate by the map $z\mapsto -z$ while no other
polynomials from $\Fc_\la$ are linearly conjugate.
Thus, if maps from $\Fc_\la$ are regarded up to linear conjugacies (preserving $0$), then the corresponding parameter space
 is the quotient of $\C$ with coordinate $b$ under the involution $b\mapsto -b$.
The \emph{principal hyperbolic component} $\Pc^\circ$ of $\Fc$ is
the set of $f_{\la,b}$ with $|\la|<1$
such that both critical points of $f_{\la,b}$ are in the Fatou component of $0$.
It is
 similar to
 the interior of the (filled) main cardioid in the (quadratic) Mandelbrot set.
On the other hand, the closure $\Pc$ of $\Pc^\circ$ has much more interesting and delicate topology than
its quadratic analog.

We study $\Pc$ through its slices $\Pc_\la=\Pc\cap\Fc_\la$
that are nonempty if and only if $|\la|\le 1$ -- an assumption
always made in this paper.
Let us define a set $\cuc\subset\Fc$
as the set of all $f=f_{\la,b}$ with $|\la|\le 1$ such that
\begin{itemize}
  \item the filled Julia set $K_f$ of $f$ has no repelling periodic cutpoints,
  \item nonrepelling periodic points of $f$ different from $0$ have multiplier 1.
\end{itemize}
The set $\cuc$ is said to be the \emph{main cubioid of $\Fc$} \cite{bopt14}.
The term is inspired by the analogy with the (filled) main cardioid of the Mandelbrot set.

By the Main Theorem of \cite{comp}, a bounded complementary component $\Uc$
of $\Pc_\la$ is \emph{stable} (see Section \ref{ss:stab}) and
for any
$f\in\Uc$, the Julia set $J_f$ of $f$ is connected, has positive measure and
carries a measurable $f$-invariant line field. One critical point of $f$ is
in the immediate (attracting or parabolic) basin of $0$ or in the Julia set
while the other one is always in the Julia set. Such stable components are
called \emph{queer} (see Section \ref{ss:cubiq}). For any compact set
$E\subset \C$ define its \emph{topological hull} $\thu(E)$ as the complement
of the unique unbounded component of $\C\sm E$; conjecturally, there are no queer
components,  and so $\thu(\Pc_\la)=\Pc_\la$.
Setting $\cuc_\la=\Fc_\la\cap\cuc$,
we have $\thu(\Pc_\la)\subset \cuc_\la$ by Theorem B of \cite{bopt14}.
Theorem \ref{t:cu-phd}
verifies a conjecture from \cite{bopt14}; it is the main result of this paper
concerning polynomial parameter spaces. By the Main Theorem of \cite{slices}
(see Theorem \ref{t:slices}), the set $\cuc_\la$ is a full continuum.

\begin{thm}
\label{t:cu-phd}
  We have $\thu(\Pc_\la)=\cuc_\la$.
\end{thm}

Theorem \ref{t:ratlam} is a dynamical application of Theorem \ref{t:nopar};
it describes the dynamics of some cubic polynomials.
From now on we abbreviate ``quadratic-like '' to ``QL''.
If $P\in\Fc_\la$ has a QL restriction
 whose filled Julia set contains $0$, then $P$ is said to be \emph{immediately renormalizable} (at $0$).
Lemma \ref{l:uniqj} relies upon Theorem 5.11 from \cite{mcm94}.

\begin{lem}[Lemma 7.2 \cite{slices}]\label{l:uniqj}
If $f$ is a complex cubic polynomial with a non-repelling fixed point $a$,
and there exists a quadratic-like filled Julia set $K^*$ with $a\in K^*$,
then $K^*$ is unique.
\end{lem}

The critical points of $P$ are denoted by $\om_1=\om_1(P)$ and
$\om_2=\om_2(P)$; they are numbered so that $\om_1\in K^*$ and $\om_2\notin K^*$ (one
omits $P$ from the notation whenever the choice of $P$ is clear) . By Lemma
\ref{l:uniqj},
this numbering of the critical points is unambiguous. Suppose that
$P\in\Fc_\la$, where $|\la|\le 1$, and that $K_P$ is connected. If $P\notin
\thu(\Pc_\la)$, then $P$ is immediately renormalizable at $0$ by
\cite[Theorem C]{BOPT16a}. Recall that some terminology and notation (e.g.,
the concept of a (parameter) wake) is introduced in the Appendix.

\begin{thm}
  \label{t:ratlam}
Consider $P\in\Fc_\la\sm\thu(\Pc_\la)$ with $|\la|\le 1$ and connected $K_P$.
Then there is a nondegenerate
paralegal cycle of cuts $\Zc$ separating $K^*$ from $\om_2(P)$.
If the vertices of $\Zc$ are parabolic then
they all equal $0$.
\end{thm}

\begin{proof}
By Theorem \ref{t:cu-phd}, we have 
$P\notin\cuc_\la$. By the Main Theorem of
\cite{slices} (see Theorem  \ref{t:slices}), the polynomial $P$ belongs to a wake
$\Wc_\la(\ta_1,\ta_2)$,
and the rays $R_P(\ta_1+1/3)$,
$R_P(\ta_2+2/3)$ land at
a repelling or parabolic point $a$;
in the latter case $a=0$. These rays and $a$ form a cut whose cycle is the desired $\Zc$.
\end{proof}

Theorem \ref{t:ratlam} provides a combinatorial framework for renormalization;
 see \cite{IK12} for a consistent combinatorial approach.

\subsection{Plan of the paper}
Section \ref{s:nopar} contains the proof of Theorem \ref{t:nopar}.
Next, in Section \ref{s:ratlam}, we prove Theorem \ref{t:cu-phd}.
Finally, Section \ref{s:apx} is the Appendix.

\section{Proof of Theorem \ref{t:nopar}}
\label{s:nopar}

Let all assumptions of Theorem \ref{t:nopar} be satisfied.

\begin{dfn}[Curve families $\Kc$ and $\Kc_\Gamma$]
  \label{d:Kc}
  Let $\Kc$ be the family of all rectifiable curves in $U_0\sm\ol U_1$ that wind once.
By Theorem \ref{t:mod-el}, we have $(\mod(U_0\sm\ol U_1))^{-1}=\el(\Kc)$.
Fig. \ref{fig:legal} shows $U_0\sm\ol U_1$ as a dark shaded annulus, and a curve $\gamma$ from $\Kc$.
For each $\Gamma=R\cup L\cup\{a\}\in\Zc$,
 let $\Kc_\Gamma$ denote the set of rectifiable curves in
 $(U_0\sm\ol U_1)\cap \Delta_{\Gamma,U_0}$ that connect $R$ with $L$.
One representative $\gamma_\Gamma\in\Kc_\Gamma$ is shown in Fig. \ref{fig:legal}.
\end{dfn}

\begin{lem}
  \label{l:FcGa}
  The family $\Kc$ overflows each of the families $\Kc_\Gamma$.
Therefore, $\el(\Kc)\ge \sum_{\Gamma\in\Zc} \el(\Kc_\Gamma)$.
\end{lem}

\begin{proof}
Take $\gamma\in\Kc$.
Connect $\d U_0$ with $\d U_1$ in $\Delta_{\Gamma,U_0}$ by an arc $\beta$ in $U_0\sm\ol U_1$, except the endpoints.
Clearly, $\beta$ must cross $\gamma$ --- otherwise $\gamma$ is contractible,
 since $U_0\sm (\ol U_1\cup\beta)$ is simply connected.
By small perturbations, arrange that $\gamma$ and $\beta$ are smooth and transverse at all intersection points.
Let $\gamma_\Gamma$ be a component of $\gamma\cap\Delta_{\Gamma,U_0}$ containing a point $b\in\beta$.
The endpoints of $\gamma_\Gamma$ are in $R\cup L$.
If both endpoints are in $R$ or both in $L$, then the intersection index
 of $\gamma_\Gamma$ and $\beta$ in $W_\Gamma$ is even.
On the other hand, the intersection index of $\gamma$ and $\beta$ is 1.
Therefore, there exists a $\gamma_\Gamma$ as above with one endpoint in $R$ and the other endpoint in $L$.
\end{proof}

Thus, we need to estimate the extremal lengths of $\Kc_\Gamma$ for all $\Gamma\in\Zc$.

\subsection{Summation trick}
\label{ss:trick}
Suppose that $V_0\subset U_0$ and $V_1=V_0\cap U_1$ are open sets such that $P:V_1\to V_0$
 is a conformal isomorphism, and $V_0$ has components $V^0$, $\dots$, $V^{m-1}$,
$m<\infty$, each containing a unique component of $V_1$.
Suppose that $P(V^i\cap U_1)=V^{i+1\pmod m}$.
Set $\Kc_i$ to be the set of rectifiable curves $\gamma$ in $U_0\sm \ol U_1$ with the following properties.
Firstly, $\gamma$ connects two boundary points of $V^i$ and is otherwise contained in $V^i$.
Secondly, $\gamma$ separates $V^i\cap U_1$ from $\d U_0$ in $\ol V^i$.
The summation trick shown below allows to estimate the sum $\sum_{i=0}^{m-1} \el(\Kc_i)$.

From now on, for any positive integer $n$, define $U_n$ inductively as the
full preimage of $U_{n-1}$ under the PL map $P:U_1\to U_0$. Define $A_i$ as
$U_i\sm\ol U_{i+1}$. The set $A_0$ is an annulus by definition of a PL map.
However, sets $A_i$ may have more complicated topology if the Julia set of
this PL map is disconnected.
Let $\Kc_{i}^*$ be the pullback of $\Kc_i$ under the homeomorphism
$P^i:A_i\cap\ol V^0\to A_0\cap\ol V^i$. Set $\Kc^*=\bigcup\Kc_i^*$. We need
Lemma \ref{l:ineq} in which we use the conventions $1/0=\infty$,
$1/\infty=0$, and $\infty+t=\infty$ for any $t\in\R_{\ge 0}$.

\begin{lem}
  \label{l:ineq}
Let $\ell$ be a given nonnegative real number. Suppose that $\ell_i$, where
$i=0$, $\dots$, $m-1$, are nonnegative numbers. Then $\sum \ell_i\ge m^2\ell$
provided $\ell=0$ or $\sum \ell_i^{-1}=\ell^{-1}$.
\end{lem}

\begin{proof}
Assume that $\ell>0$ (the case $\ell=0$ is obvious). Setting
$x^2_i=\ell_i^{-1}$ for $0\le i\le m-1$,
 the desired inequality can be restated as $(x_1^2+\dots+x_m^2)(x_1^{-2}+\dots+x_m^{-2})\ge m^2$,
which is the Cauchy--Schwarz inequality.
Alternatively, the lemma reduces to a classical inequality between the arithmetic mean and the
 harmonic mean.
\end{proof}

\begin{prop}
  \label{p:trick}
  We have $\sum_{i=0}^{m-1} \el(\Kc_{i})\ge m^2\el(\Kc^*)$
\end{prop}

\begin{proof}

Since $P^i$ is a conformal univalent map on $A_i\cap V^0$, then
$\el(\Kc_i^*)=\el(\Kc_i)$. The families $\Kc_i^*$ are disjoint. Hence, the
Parallel Law (Theorem \ref{t:parlaw}) applies to $\Kc^*=\bigcup\Kc_i^*$. Set
$\ell=\el(\Kc^*)$ and $\ell_i=\el(\Kc_i^*)$. By the Parallel Law, $\sum
\ell_i^{-1}=\ell^{-1}$, and the claim follows from Lemma \ref{l:ineq}.
\end{proof}

\subsection{Fatou accesses}
\label{ss:Facc} Consider a periodic access $\alpha$ of period $m_\al$ from a Fatou domain
$\Omega$ to a periodic point $a\in J_P$. 
Let $O(a)$ be a small disk neighborhood of $a$. Then
$f=P^{m_\al}|_{O(a)}$ is univalent. Identify two points of $\Omega\cap
(O(a)\sm\{a\})$ if they belong to the same $f$-orbit. Let $\pal$ 
be the quotient map. The component of the quotient space containing
$\pal(\alpha)$ is an annulus denoted by $\A_\alpha$.
Write $\la^*_\alpha$ for the conjugate multiplier of $\alpha$. 

\begin{lem}[Proposition 4.3 of \cite{Pet93}]
  \label{l:modAal}
  The modulus of the annulus $\A_\alpha$ is $\pi/\log\la^*_\alpha$.
In particular, $\mod(\A_\alpha)=\pi/(m_\alpha \log D)$ if $\alpha$ is in $\C\sm K_P$.
\end{lem}

If $a$ is the vertex of $\Gamma\in\Zc$,
and a periodic access $\alpha$ to $a$ is in a $\Zc$-wedge $W_\Gamma$, then $\al\subset \Omega$
where $\Omega\subset W_\Gamma$ is a bounded $\Zc$-domain
(and $\alpha$ is determined by $\Omega$)
or $\Om=\C\sm K_P$ (and
there may be many different accesses to $a$ from $\Om\cap W_\Ga$). Let $\Kc_\alpha$ be
the collection of rectifiable curves $\gamma$ in $U_0\sm\ol U_1$ that
(1) 
connect two boundary points of $\Omega$ and otherwise lie in $\Omega$, and
(2) 
separate $\alpha$
from $\d U_0\cap\Omega$ in $\Omega$.

\begin{lem}
  \label{l:FcOm}
  The collection $\Kc_\alpha$ is nonempty, that is, $\Omega$ cannot lie entirely in $U_0$.
Moreover, $\Kc$ overflows $\Kc_\alpha$.
\end{lem}

\begin{proof}
If $\Omega\subset U_0$, then $\d U_{m_\al}$ is disjoint from $\Om$
as otherwise $\Om\cap \d U_0\ne \0$; hence $\Om\subset
U_{m_\al}$. Repeating this, we see that $\Om\subset K^*$, a contradiction.
The last claim of the lemma
is immediate (cf. Lemma  \ref{l:FcGa}).
\end{proof}

Finally, we  use the summation trick of Section \ref{ss:trick} to estimate the total contribution of all $\Kc_{P^i(\alpha)}$.

\begin{lem}
  \label{l:Fcal}
For $\al$ as above $\displaystyle{
\sum_{i=0}^{m_\alpha-1}\el(\Kc_{P^i(\alpha)})\ge \frac{m_\alpha^2\pi}{\log\la^*_\alpha}=
\sum_{i=0}^{m_\alpha-1}\frac{m_\alpha\pi}{\log\la^*_{P^i(\alpha)}}.}
$
\end{lem}

Note: it is not claimed that $\el(\Kc_{P^i(\alpha)})\ge \frac{m_\alpha\pi}{\log\la^*_{P^i(\alpha)}}$.
However, this estimate holds true after averaging over the cycle of $\alpha$.

\begin{proof}
Apply the summation trick to $V_1$ defined as the union of components of
 $U_1\cap\bigcup_{i=0}^{m_\al-1} P^i(\Omega)$ attached to $\Zc$-vertices and $V_0=P(V_1)$.
Note that $V_0$ has $m_\alpha$ components even though
$P^i(\Omega)$ are the same if $\Omega$ is the basin of infinity.
By Proposition \ref{p:trick}
$$
\sum_{i=0}^{m_\al-1}\el(\Kc_{P^i(\alpha)})\ge m_\alpha^2\,\el(\Kc^*)
$$
As $\pal$ is injective on $A_m$ and $\pal(\Kc^*)\subset \Kc(\A_\alpha)$
(where $\Kc(\A_\alpha)$ is as in Definition \ref{d:families}),
then $\el(\Kc^*)=\el(\pal(\Kc^*))\ge\mod(\A_\alpha)=\pi/\log\la^*_\alpha$.
The inequality of the lemma follows.
The equality holds because $\la^*_{P^i(\alpha)}=\la^*_\alpha$.
\end{proof}

\subsection{Side annuli}
\label{ss:sideann}
A periodic access $\al$ from $\C\sm K_P$ to a periodic point $a\in K_P$  
corresponds to a unique periodic external ray $X_\al$ of $P$ with
$P^{m_\alpha}(X_\al)=X_\al$ (conversely, given an external ray $Y$ landing at
$y\in J_P$ denote by $\al_Y$ the corresponding access). The set $\A_\alpha\sm
\pal(X_\al)$ consists of two \emph{side annuli} $\A_\alpha^+$ and
$\A_\alpha^-$ (i.e., $X_\al$ divides $\C\sm K_P$ locally near
$a$ into two sectors projecting to $\A_\alpha^\pm$ 
by $\pal$). Choose the labeling so that $\A_\alpha^+$ corresponds to the
positive (counterclockwise) side of $X_\al$. The image of $X_\al$ in
$\A_\alpha$ is the unique simple closed Poincar\'e geodesic
 (cf. the proof of Theorem I.A in \cite{Hub93}).
It divides $\A_\alpha$ into two annuli of 
modulus $\mod(\A_\alpha)/2$, by Lemma \ref{l:hann}.

\begin{lem}
  \label{l:hann}
  Let $A$ be a topological annulus, and let $\gamma\subset A$ be the unique simple closed geodesic in $A$.
Then $A\sm\gamma$ consists of two annuli, each of modulus $\mod(A)/2$.
\end{lem}

Lemma \ref{l:hann} is well-known; 
 since $A$ is isomorphic to the flat cylinder $\{|\Im(z)|<h/2\}/\Z$, where $h=\mod(A)$,
 the statement follows immediately from the reflection symmetry of the cylinder ($\gamma$ is then represented by $\R/\Z$),
 cf. Remark 2.41 and Section 2.6.1 of \cite{lyu-book}.

For a $\Zc$-access $\al$ let $\Kc_\al^+$ (resp., $\Kc_\al^-$) be the family of rectifiable curves $\ga$ in $U_0\sm\ol U_1$ 
with $\pal(\ga)\in\Kc(\A_\alpha^+)$ (resp., $\pal(\ga)\in\Kc(\A_\alpha^-)$.

\begin{lem}
  \label{l:Fcplus}
We have
 $\sum_{i=0}^{m_\alpha-1} \el(\Kc_{P^i(\al)}^+)\ge \frac{m_\alpha\pi}{2\log D}$, and similarly for $\Kc_\al^-$.
\end{lem}

The proof is similar to that of Lemma \ref{l:Fcal}. We use the fact that
the families $\Kc_\al^\pm$ identify with certain subfamilies of $\Kc(\A_\alpha^\pm)$, as well as
the summation trick of Section \ref{ss:trick} for properly chosen $V_0$, $V_1\subset\C\sm K_P$.
Observe that
since $\la^*_\alpha=m_\alpha\log D$, then
there is $m_\alpha$ rather than $m_\alpha^2$ in the numerator.

\subsection{Proof of Theorem \ref{t:nopar}}
For a cut $\Gamma=R\cup L\cup\{a\}\in\Zc$ of 
period $m$ and its wedge $W=W_\Ga$, let $\Omega_1$, $\dots$, $\Omega_t$ be all bounded
$W$-domains.
Write $\lambda_j^*$ for the corresponding conjugate multipliers ($j=1,\dots,t$).
Let $k$ be the number of external rays 
in $W$ landing at $a$; clearly, $k+1\ge t$.

\begin{prop}
  \label{p:invcut}
In the above situation we have
$$
\sum_{i=0}^{m-1}\el(\Kc_{P^i(\Gamma)})\ge \frac{m(k+1)\pi}{\log D}+m\sum_{j=1}^{t}\frac{m\pi}{\log\la^*_j}=
$$
$$
 =\frac{m(k+1)\pi}{\log D}+\sum_{i=0}^{m-1}\sum_{j=1}^{t}\frac{m\pi}{\log\la^*_j}.
$$
\end{prop}

It follows that the average (over time) contribution of each $\Omega_j$ is at least $m\pi/\log\la^*_j$.
If $a$ is not accessible from any bounded Fatou component in $W$, then the second term in the right hand side is zero.

\begin{proof}
For every $j=1$, $\dots$, $t$, there is a single access $\alpha_j$ from $\Omega_j$ to $a$.
Also, let $\gamma_1$, $\dots$, $\gamma_k$ be the accesses to $a$ from $W\sm K(P)$ corresponding
 to the external rays in $W$ landing at $a$.
Recall that $\alpha_R$ and $\alpha_L$ are accesses represented by $R$ and $L$, respectively.
For every $i=0$, $\dots$, $m-1$, we have
$$
\el(\Kc_{P^i(\Gamma)})\ge \sum_{j=1}^{t}\el(\Kc_{P^i(\alpha_j)})+\sum_{l=1}^{k}\el(\Kc_{P^i(\gamma_l)})
+\el(\Kc^-_{P^i(\al_R)})+\el(\Kc^+_{P^i(\al_L)})
$$
since $\Kc_{P^i(\Gamma)}$ overflows each of the disjoint families $\Kc_{P^i(\alpha_j)}$,
 $\Kc_{P^i(\gamma_l)}$, $\Kc^-_{P^i(\al_R)}$, and $\Kc^+_{P^i(\al_L)}$.
(According to our orientation conventions, $\Kc^-_{\al_R}$ and $\Kc^+_{\al_L}$ both lie in $W$.)
Taking the sum of both parts as $i$ runs from $0$ to $m-1$ and applying Lemmas \ref{l:Fcal} and \ref{l:Fcplus}
we obtain the desired inequality.
\end{proof}

In order to conclude the proof of Theorem \ref{t:nopar},
 use Lemma \ref{l:FcGa} and split the right hand side into blocks corresponding to different cycles of cuts.
For every cycle of cuts, use Proposition \ref{p:invcut}.
Theorem \ref{t:nopar} follows.

\section{Applications}
\label{s:ratlam}
We prove Lemma \ref{l:sibling-inside} that is used in Theorem \ref{t:bound-period},
 and Theorem \ref{t:cu-phd}.

\subsection{A sufficient condition of being legal}
\label{ss:legal} Let $f:\C\to \C$ be a branched covering, and $U\subset\C$ be an
open Jordan disk.
A \emph{closed ray} $R\subset \C$ is the image of $\R_{\ge 0}=\{x\mid
x\ge 0\}$ under an embedding $g$ such that $g(t)\to \infty$ as $t\to \infty$,
and an \emph{open ray} is a closed ray with the $g$-image of $0$ removed; in
either case the $g$-image of $0$ is called the \emph{endpoint} of $R$.
Finally, the union of finitely many rays that share an endpoint and are
otherwise disjoint is called a \emph{non-compact star}, and the common
endpoint of the rays forming the star is called the \emph{vertex} thereof.

\begin{lem}
  \label{l:cell}
Let $f:\C\to \C$ be a branched covering, and $U\subset\C$ be an open Jordan disk.
Suppose that $f(U)$ is a Jordan disk, there are no critical values of $f$ in $f(\d U)$,
 and every component $V$ of $f(U)\sm f(\d U)$
 contains at most $|\d V\cap\d f(U)|$ critical values of $f$.
If $f:U\to f(U)$ is not a homeomorphism, then $f$ has a critical point in $U$.
\end{lem}

\begin{figure}
  \centering
  \includegraphics[width=\textwidth]{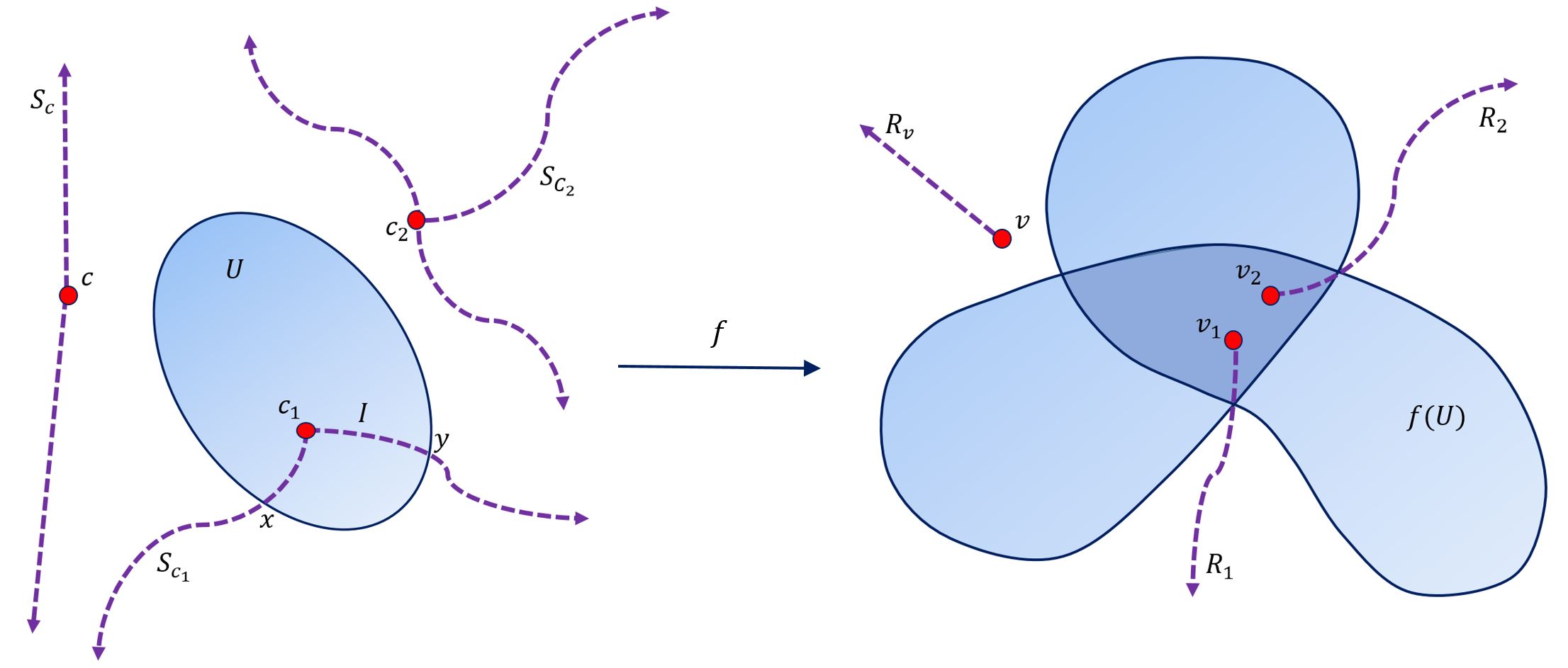}
  \caption{\small An illustration to Lemma \ref{l:cell}.
  Right: the set $f(U)$, two rays $R_1$, $R_2$ originating at critical values $v_1$, $v_2\in f(U)$,
  and a ray $R$ originating at a critical value $v\notin \ol{f(U)}$. Note: each of the rays $R_1$, $R_2$
  crosses $f(\d U)$ only once, not counting multiplicities.
  Left: the set $U$, non-compact stars $S_{c_1}$, $S_{c_2}$, $S_c$, and a segment $I\subset S_{c_1}$
  with endpoints $x$ and $y$.}
  \label{fig:L31}
\end{figure}

Note that, as there are no critical values in $f(\d U)$, there are no critical values in $\d f(U)\subset f(\d U)$.
Even though the assumption of Lemma \ref{l:cell} can be somewhat relaxed, it cannot be dropped altogether,
 see Fig. \ref{fig:mis}.

\begin{proof}
Connect the critical values of $f$ to
infinity with pairwise disjoint closed rays
 as described below.
Let $v_1, \dots, v_k\in f(U)$ be all critical values in $f(U)$.
Then, by
the assumption, each $v_j$  for $j=1,\dots,k$ can be connected to infinity with a closed
ray $R_{v_j}=R_j$ so that
\begin{enumerate}
  \item sets $R_j\cap \ol{f(U)}$ are simple arcs avoiding $f(\d U)\cap f(U)$,
  \item sets $R_j$ are pairwise disjoint, and
  \item rays $R_j$
contain no critical values other than their endpoints $v_j$.
\end{enumerate}
Then connect
each other critical value $v\notin f(U)$ to infinity with a ray $R_v$ so that
$R_v$ is disjoint from $f(U)$ and all the rays are pairwise disjoint.
This is possible because $f(U)$ is a disk. 

For every critical point $c$ of $f$, let $S_c$ be the pullback of $R_{f(c)}$ containing $c$.
Clearly, $S_c$ is a non-compact star, and $S_c\cap S_{c'}$ for $c\ne c'$.
The stars $S_c$ partition $\C$ into
open pieces each of which maps onto its image homeomorphically.
Since $f:U\to f(U)$ is not a homeomorphism, $U$ must cross some $S_c$.
Thus there is a closed arc $I\subset S_c$ with endpoints $x$, $y$
 such that $x$, $y\in \d U$ and $I\sm \{x, y\}=I^\circ\subset U$.
This implies that $f(I^\circ)\subset f(U)$ and, by construction,
 $f(I^\circ)\subset R_j$ for some $j=1,\dots,k$.
Since $R_j\cap\d f(U)$ is only one point, we must have $f(x)=f(y)$.
Therefore, $c\in I^\circ$ as claimed.
\end{proof}

\begin{figure}
  \centering
  \includegraphics[width=8cm]{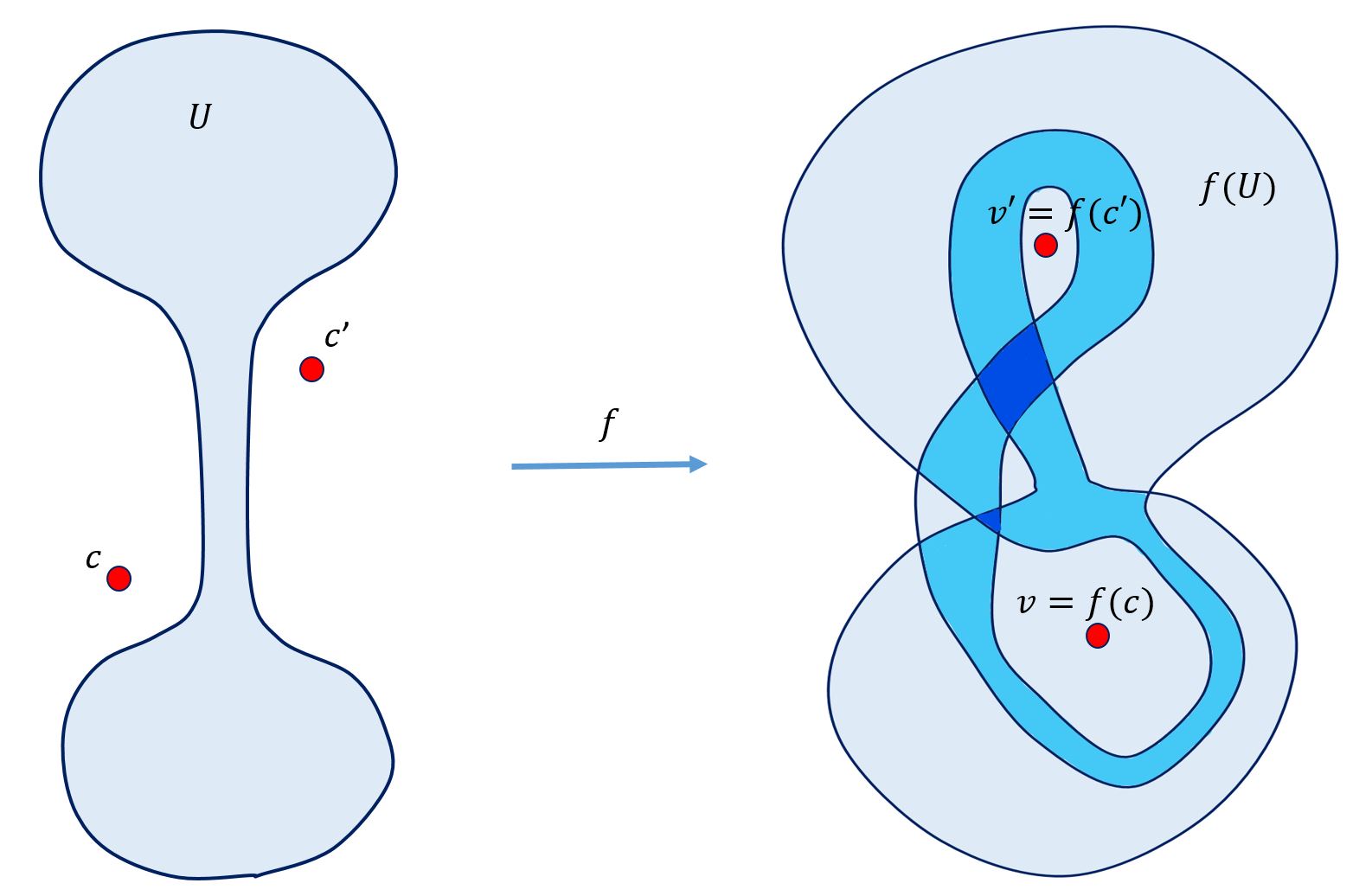}
  \caption{\small An example (due to Micha\l{} Misiurewicz) of a branched covering $f:\C\to\C$ and
  a topological disk $U\subset\C$ such that $f(U)$ is a disk, but there are no critical points of $f$ in $U$.
  Note that an assumption of Lemma \ref{l:cell} is violated: namely, the component $V$ of $f(U)\sm f(\d U)$
  containing the critical value $v=f(c)$ does not touch $\d f(U)$, and similarly with the component $V'\ni v'=f(c')$.
  The critical points $c$, $c'$ lie outside of $U$.}
  \label{fig:mis}
\end{figure}

Say that $U_0$ is \emph{$\Gamma$-adapted} if $\d U_0$ is smooth and transversal to
the cut $\Gamma$ so that $\Ga\cap \d U_0$ is finite. One can arrange that
$U_0$ is $\Gamma$-adapted by a small perturbation.

\begin{lem}\label{l:sibling-inside}
For every cut $\Gamma=R\cup L\cup\{a\}$ attached to $K^*$ and such that $P'(a)\ne 0$,
 the restriction of $P$ to the core component $\Delta_{\Gamma,U_1}$ is univalent.
Thus, an invariant set of cuts is paralegal if all its cuts are attached to $K^*$.
\end{lem}

\begin{proof}
By a small perturbation, we may assume that $U_0$ is $P(\Gamma)$-adapted.
Set $\Delta_1=\Delta_{\Gamma,U_1}$ and $\Delta_0=\Delta_{P(\Gamma),U_0}$.
Components of $P(\Gamma)\cap U_0$ are homeomorphic to the interval $(0,1)$; call them
 \emph{pseudo-chords} of $U_0$.
Define \emph{pseudo-gaps} of $U_0$ as the complementary 
 components of $U_0$ to the union of pseudo-chords.

The image $P(\Delta_1)$ may be bigger than $\Delta_0$ but it is a union of
 pseudo-chords and pseudo-gaps of $U_0$.
By the Riemann mapping theorem, $U_0$ is isomorphic to the open unit disk $\disk$.
The corresponding partition of $\disk$ can be straightened in the following sense:
 replace the pullback of every pseudo-chord with a straight chord connecting the same boundary points of $\disk$.
The thus obtained chords are clearly disjoint.
Now, $\disk$ is partitioned into \emph{chords} and \emph{gaps} corresponding to
 the pseudo-chords and pseudo-gaps of $U_0$.
Consider the union of chords and gaps of $\disk$ corresponding to $P(\Delta_1)$.
This set is necessarily convex, hence homeomorphic to $\disk$.
We conclude that $P(\Delta_1)$ is also a Jordan disk.

Suppose that there is a critical point $c$ of $P$ in $\Delta_1$.
Connect $P(c)$ with $P(a)$ by an arc $\beta$ disjoint from $P(\Gamma)$
 except for the endpoint $P(a)$.
There are several ($>1$) pullbacks of $\beta\sm\{P(c)\}$ with endpoint $c$.
On the other hand, two different pullbacks of $\beta\sm\{P(c)\}$
 cannot connect $c$ with $a$ --- since $P$ is injective near $a$.
Choose a pullback $\alpha$ whose other endpoint $b$ is different from $a$.
By definition, $b\in K^*$.
Since $\beta$ is disjoint from $P(\Gamma)$ except for endpoint $P(a)$,
 the arc $\alpha$ lies entirely in $W_\Gamma$.
Therefore, $b\in W_\Gamma$, a contradiction with $\Gamma$ being attached to $K^*$.
The desired statement now follows from Lemma \ref{l:cell} applied to $P:\Delta_1\to P(\Delta_1)$.
Since every component of $P(\Delta_1)\sm P(\d\Delta_1)$ contains arcs of
 $\d P(\Delta_1)$ on the boundary, the assumptions of Lemma \ref{l:cell} are fulfilled.
\end{proof}

\subsection{Proof of Theorem \ref{t:cu-phd}}
\label{ss:proof-cu-phd}
Lemma \ref{l:pl-depend-1} below is based on Theorem \ref{t:bound-period}.

\begin{lem}\label{l:pl-depend-1}
Let $P:U_1\to U_0$ be a PL map with
no critical points in $\d U_1$.  Set $\mod(U_0\sm\ol{U}_1)=\mu$. If $P_n\to
P$ is a sequence of polynomials and $U^n_1\ni 0$ is a $P_n$-pullback of $U_0$
for any $n$, then $P_n:U^n_1\to U_0$ is PL for large $n$, and any cycle of
periodic cuts attached to $K^*_n$
has period at most $\frac{\log D}{\mu\pi}$.
\end{lem}

\begin{proof}
Since $P_n\to P$, then for any $\eps>0$ there is $N=N_\eps$ such that if
$n>N$ then $P_n:U^n_1\to U_0$ is PL and an annulus $A$ of modulus
$\mod(A)>\mu-\eps$ is essentially embedded into $U_0\sm\ol{U^n_1}$ so that
$\mod(U_0\sm\ol{U^n_1})\ge\mu-\eps>0$.
By Theorem \ref{t:bound-period}, we have $m\le
\frac{\log D}{(\mu-\eps)\pi}$ for a period $m$ cycle of cuts attached to
$K^*(P_n)$, $n>N$. 
Choosing $\eps$, we guarantee that
$\frac{\log D}{(\mu-\eps)\pi}$ is less than the integer part of $\frac{\log
D}{\mu\pi}$ plus $\frac12$ which implies the desired.
\end{proof}

To prove Theorem \ref{t:cu-phd}, it suffices to show that any
$P\in\Cc_\la\sm\thu(\Pc_\la)$ is outside of $\cuc_\la$. 
Such $P$ is immediately renormalizable by
Theorem \ref{t:holmot} from the Appendix. 
The corresponding filled Julia set $K^*$ is connected. The two critical
points of $P$ are $\om_1(P)\in K^*$ and $\om_2(P)\notin K^*$. There are two
cases to consider: either the critical point $\om_2=\om_2(P)$ is active or it
is passive. The former case is considered in the following proposition.

\begin{prop}
  \label{p:active}
Take $P\in\Cc_\la\sm\thu(\Pc_\la)$. If $\om_2(P)$ is active and $P$ is not
the root point of a wake of $\Cc_\la$, then $P\notin\cuc_\la$.
\end{prop}

The proof of Proposition \ref{p:active} is based on Lemma \ref{l:active}, which implements a standard normal family argument.
A similar claim for Misiurewicz rather than critically periodic parameters is given in Proposition 2.1 of \cite{Mc00}.

\begin{lem}
  \label{l:active}
  Under assumptions of Proposition \ref{p:active}, there is a sequence $P_n\in\Cc_\la\sm\thu(\Pc_\la)$ converging to $P$ and
such that $\om_2(P_n)$ is $P_n$-periodic.
\end{lem}

\begin{proof}
If $\Uc$ is a Jordan disk neighborhood of $P$ in $\Fc_\la$ disjoint from
$\thu(\Pc_\la)$, then $\om_2(P)$ is well defined for all $P\in\Uc$ and
depends holomorphically on $P$. Note that $\om_1(P)$ is never mapped to
$\om_2(P)$. By way of contradiction, assume that $\om_2(P)$ is not periodic
for all $P\in\Uc$. Then the backward orbit of $\om_2(P)$ moves
holomorphically with $P\in\Uc$. Choose three distinct elements $a(P)$,
$b(P)$, $c(P)$ from this backward orbit. Since the functions $P\mapsto
P^n(\om_2(P))$ do not form a normal family on $\Uc$,
 they cannot avoid the 
 points $a(P)$, $b(P)$, $c(P)$.
Thus there is a $P\in\Uc$ and $n$ such that $P^n(\om_2(P))$ coincides, say,
with $a(P)$, which implies that $\om_2(P)$ is periodic.
\end{proof}

\begin{proof}[Proof of Proposition \ref{p:active}]
Choose $P_n\to P$ as in Lemma \ref{l:active}. Since $\om_2(P_n)$ is periodic
and by definition of $\cuc$, it follows that $P_n\notin\cuc_\la$. By
Theorem \ref{t:slices}, the polynomial $P_n$ lies in a wake $\Wc_\la(\ta^n_1,\ta^n_2)$,
 where $\ta^n_1+1/3$ and
$\ta^n_2+2/3$ are periodic. Let $a_n$ be the common landing point of the rays
$R_{P_n}(\ta_1^n+1/3)$ and $R_{P_n}(\ta_2^n+2/3)$. Write $\Lambda_n$ for the
cut formed by these two rays and $a_n$.
If infinitely many of $P_n$ are in the same wake, then $P$ is in this wake too,  and hence $P\notin\cuc_\la$
(the only point of $\cuc_\la$ in the closure of $\Wc_\la(\ta_1,\ta_2)$ is the root point of $\Wc_\la(\ta_1,\ta_2)$, and $P$ is not
that root point by the assumption).
Passing to a subsequence, we may assume that all pairs $\{\ta_1^n,\ta_2^n\}$ are different.
Also, we may assume that all $a_n$ are repelling, since there are only finitely many wakes associated
 with the parabolic vertex $0$ (see Theorem \ref{t:slices}).

Let $P:U_1\to U_0$ be a QL restriction of $P$ with connected QL filled Julia
set $K^*$. Replacing $U_1$ and $U_0$ with smaller disks if necessary, we may
assume that
 there are no critical points on the boundary of $U_1$.
Set $U^n_1$ to be the component of $P_n^{-1}(U_0)$ containing $0$.
Since $P_n\to P$ and by Lemmas \ref{l:trivial}, \ref{l:pl-depend}, the map
$P_n:U^n_1\to U_0$ is a QL map if $n$ is large.
Moreover, all $P_n$'s have connected filled Julia sets $K^*_n$,
 near which they are hybrid equivalent to $Q_\la(z)=\la z+z^2$.
It now follows from Lemma \ref{l:pl-depend-1} that the cuts $\Lambda_n$ have bounded periods.
This is a contradiction, since there are only finitely many wakes of any given period.
(Recall that the period of the wake $\Wc_\la(\ta^n_1,\ta^n_2)$ is defined as the period of the cut $\Lambda_n$.)
\end{proof}

We can now complete the proof of Theorem \ref{t:cu-phd}.

\begin{proof}[Proof of Theorem \ref{t:cu-phd}]
By way of contradiction, 
let $\cuc_\la\sm\thu(\Pc_\la)\ne\0$.
Since $\cuc_\la$ is a full continuum, and $\thu(\Pc_\la)$ is compact,
 there are uncountably many boundary points of $\cuc_\la$ outside of $\thu(\Pc_\la)$.
Choose a boundary point $P$ of $\cuc_\la$ so that $P$ is not in $\thu(\Pc_\la)$ and not a root point of a wake.
Such $P$ exists since there are only countably many wakes, hence they have only countably many root points altogether.
The critical point $\om_2(P)$ is active since $P\in\d\cuc_\la$.
Theorem \ref{t:cu-phd} now follows from Proposition \ref{p:active}.
\end{proof}

\section{Appendix: background material}
\label{s:apx}
This section gives an overview of known results used in the paper including classical foundations
 as well as more recent specific developments.

\subsection{Moduli and extremal length}
\label{ss:mod}
Most of results in Section \ref{ss:mod} can be found in classical textbooks, e.g., \cite{ahl66}.

Let $A$ be a Riemann surface homeomorphic to an annulus.
Then, by the Uniformization Theorem, there is a conformal isomorphism between $A$ and a Euclidean cylinder
 of height $\mu$ and circumference $1$.
In this case, $\mu$ is called the \emph{modulus} of $A$ and is denoted by
$\mod(A)$. This is a conformal invariant. It is a straightforward computation
using the complex logarithm function that the modulus of the round annulus
$A=\{z\in\C\mid r_1<|z|<r_2\}$ is given by $\mod(A)=\log(r_2/r_1)/(2\pi)$.

\begin{dfn}[Extremal length]
\label{d:el}
  Let $\Kc$ be a family of locally rectifiable curves in $\C$ or in a Riemann surface.
The \emph{extremal length} of $\Kc$ is defined as
$$
\el(\Kc)=\sup_\rho\frac{L_\rho(\Kc)^2}{\mathrm{area}(\rho)}.
$$
Here $\rho$ ranges through all measurable conformal metrics on $\C$ (or on the chosen Riemann surface)
 of finite positive area $\mathrm{area}(\rho)$,
 and $L_\rho(\Kc)$ is the infimum length of a curve from $\Kc$ with respect to $\rho$.
\end{dfn}

\begin{dfn}[Overflow]\label{d:overf}
For two families of curves $\Kc_1$ and $\Kc_2$ we say that $\Kc_2$
\emph{overflows} $\Kc_1$ and write $\Kc_1<\Kc_2$ if every curve from $\Kc_2$
is an extension of a curve from $\Kc_1$.
\end{dfn}

Informally, $\Kc_1<\Kc_2$ means that $\Kc_2$ has \emph{fewer} curves that are
\emph{longer}. Note that $\Kc_2\subset\Kc_1$ implies $\Kc_1<\Kc_2$
(``reversion of the inequality''!). The next proposition follows immediately
from definition.

\begin{prop}
  \label{p:ELmon}
If $\Kc_1<\Kc_2$, then $\el(\Kc_1)\le\el(\Kc_2)$.
\end{prop}

\begin{dfn}\label{d:families}
For an open annulus $A\subset \C$ let $\Kc(A)$ be
the family of all rectifiable curves in $A$ connecting the
boundary components of $A$, and let $\Kc_\circ(A)$ be the family  of all
closed rectifiable curves that wind once in $A$.
\end{dfn}

For the following classical result see, e.g., \cite{ahl66}.

\begin{thm}
  \label{t:mod-el}
We have $\displaystyle{\mod(A)=\el(\Kc(A))=\frac{1}{\el(\Kc_\circ(A))}.}$
\end{thm}

Now recall the parallel and series laws for extremal lengths, cf. the
Appendix in \cite{KL09a}. Two families of curves $\Kc_1$, $\Kc_2$ are
\emph{disjoint} if any curve from $\Kc_1$ is disjoint from any curve from
$\Kc_2$.

\begin{thm}[Parallel Law]
\label{t:parlaw}
 Suppose that $\Kc_1$, $\dots$, $\Kc_m$ are pairwise disjoint families of rectifiable curves in $\C$.
Then
$$
\frac{1}{\el\left(\Kc_1\cup\dots\cup\Kc_m\right)}=\sum_{i=1}^{m}\frac{1}{\el(\Kc_i)}.
$$
\end{thm}

\begin{thm}[Series Law]
\label{t:serlaw}
  Suppose that $\Kc_1$, $\dots$, $\Kc_m$ are pairwise disjoint families of rectifiable curves in $\C$.
If a family $\Kc$ of rectifiable curves overflows each of the families $\Kc_1$, $\dots$, $\Kc_m$, then
$$
\el(\Kc)\ge \sum_{i=1}^{m} \el(\Kc_i).
$$
\end{thm}

The Series Law is essentially equivalent to the \emph{Gr\"otzsch inequality} on the moduli of annuli.
An annulus $A_2$ is \emph{essentially embedded} into an annulus $A_1$ if $A_2\subset A_1$,
 and the identical embedding of $A_2$ into $A_1$ induces an isomorphism of fundamental groups.

\begin{lem}[Gr\"otzsch inequality]
  \label{l:Grotzsch}
  If $A_1$, $\dots$, $A_n$ are pairwise disjoint annuli essentially embedded into an annulus $A$, then
  $$
  \mod(A)\ge \mod(A_1)+\dots+\mod(A_n).
  $$
\end{lem}

\subsection{External rays}
\label{ss:extray}
Let $P$ be a degree $D>1$ complex polynomial.
The \emph{filled Julia set} $K_P$ is
the set $\{z\in\C\mid P^n(z)\not\to\infty\}$. This is a nonempty compact set;
the \emph{Julia set} $J_P$ is its boundary $\d K_P$. A classical
theorem of B\"ottcher states that $P$ is conjugate to $z\mapsto z^D$ near
infinity. If $K_P$ is connected, then the conjugacy extends to a conformal
isomorphism between $\ol\C\sm K_P$
 and the open unit disk $\disk=\{z\in\C\,|\, |z|<1\}$.
Without loss of generality we may assume that $P$ is \emph{monic}, i.e., the highest term of $P$ is $z^D$.
Then there is a conformal isomorphism $\psi_P:\disk\to\ol\C\sm K_P$ conjugating $z\mapsto z^D$
 with $P$ and normalized so that $\psi_P(0)=\infty$ and $\psi'_P(0)>0$.
The inverse map $\phi_P=\psi_P^{-1}$ is called the \emph{B\"ottcher coordinate}.
An \emph{external} ray $R_P(\theta)$ of \emph{argument} $\theta\in\R/\Z$ is
the $\psi_P$-image of
 $\{e^{2\pi i\theta}\rho\,|\,\rho\in (0,1)\}$; clearly, $P(R_P(\theta))=R_P(D\theta)$.

 A ray $R_P(\theta)$ \emph{lands} at $a\in K_P$ if $a=\lim_{\rho\to 1^-} \psi_P(e^{2\pi i\theta}\rho)$
is the only accumulation point of $R_P(\theta)$ in $\C$.
By the Douady--Hubbard--Sullivan landing theorem, if $\theta$ is rational,
 then $R_P(\theta)$ lands at a (pre)periodic point that is eventually mapped to a repelling or parabolic periodic point.
A periodic point $a$ with $P^m(a)=a$ is \emph{repelling} if $|(P^m)'(a)|>1$ and
 \emph{parabolic} if $(P^m)'(a)$ is a root of unity.
Conversely, any point that eventually maps to a repelling or parabolic periodic point is the
 landing point of at least one and at most finitely many external rays with rational arguments.

\subsection{Polynomial-like maps}
\label{ss:pl} Let $U$ and $V$ be Jordan disks such that $U\Subset V$ (i.e.,
$\ol{U}\subset V$). The following classical definition is due to Douady and
Hubbard \cite{DH-pl}. A proper holomorphic map $f:U\to V$ is said to be
\emph{polynomial-like} (\emph{PL}); if the degree of $f$ is two it is called
\emph{quadratic-like} (\emph{QL}). The \emph{filled Julia set} $K(f)$ of $f$
is defined as the set of points in $U$, whose forward $f$-orbits stay in $U$.
Similarly to polynomials, the set $K(f)$ is connected if and only if all
critical points of $f$ are in $K(f)$.

 Let $f_1:U_1\to V_1$ and $f_2:U_2\to V_2$ be two PL maps.
Consider Jordan neighborhoods $W_1$ of $K(f_1)$ and $W_2$ of $K(f_2)$.
A quasiconformal\footnote{i.e., such that $\exists C>0$ with the property
$C^{-1}\mod(A)\le \mod(\phi(A))\le C\,\mod(A)$ for any annulus $A\subset W_1$.}
 homeomorphism $\phi:W_1\to W_2$ is called a \emph{hybrid equivalence} between $f_1$ and $f_2$
 if $f_2\circ\phi=\phi\circ f_1$ whenever both parts are defined, and  $\ol\d\phi=0$ on $K(f_1)$.
By the Straightening Theorem of \cite{DH-pl}, a PL map $f :U\to V$
 is hybrid equivalent to a polynomial of the same degree restricted on a Jordan neighborhood of its filled Julia set.
(Abusing the language, we will simply say ``hybrid equivalent to a
polynomial''.)

Lemma \ref{l:trivial} easily follows from \cite{DH-pl}, \cite{DH84} and Lemma
\ref{l:uniqj}.

\begin{lem}\label{l:trivial}
If $P\in \Fc_\la, |\la|\le 1$ has a QL
restriction $P:U^*\to V^*$ with $0\in U^*$, then the corresponding filled QL
Julia set $K^*$ is connected and unique; also, $P$ and the map
$Q_\la(z)=\la z+z^2$ are hybrid equivalent near their (QL) filled Julia sets.
\end{lem}

\subsection{Stability}
\label{ss:stab}

We need the following topological lemma.

\begin{lem}\label{l:topol}
Let $P$ be a polynomial and $P:U_1\to U_0$ be a PL map of degree $d$. Assume
that there are no critical points in $\d U_1$.
Suppose that 
a sequence of PL maps $P_n:U^n_1\to U_0$ with $U^n_1\cap U_1\ne \0$ converges to $P$ as $n\to\infty$.
Then $U^n_1$ are $P_n$-pullbacks of $U_0$ and
 the degree of $P_n|_{U^n_1}$ is $d$ for all sufficiently large $n$.
\end{lem}

\begin{proof}
Since $P_n:U^n_1\to U_0$ is a PL map, then
$U^n_1$ is a $P_n$-pullback of $U_0$. Since $\d U_1$ contains no critical
points of $P$, there is a unique $P_n$-pullback $U^n_1$ of $U_0$ non-disjoint
from $U_1$, and all other $P_n$-pullbacks of $U_0$ are positively distant from $U^n_1$.
The rest of the lemma follows.
\end{proof}

The next lemma is more specific for our setup.

\begin{lem}\label{l:pl-depend}
Let $P$ be a polynomial, $P:U_1\to U_0$ be a PL map of degree $d$ with
connected filled PL Julia set $K^*$ and no critical points in $\d U_1$.
Suppose that $P_n\to P$ is a sequence of polynomials such that for some
Jordan disks $U^n_1\Subset U_0$ non-disjoint from $U_1$ the maps
$P_n:U_1^n\to U_0$ are PL with connected PL Julia sets. Then the PL maps
$P_n:U_1^n\to U_0$ are of degree $d$ for large $n$,
 and their filled PL Julia sets $K^*_n$ converge into $K^*$.
\end{lem}

\begin{proof}
Follows from Lemma \ref{l:topol}, continuity and the definitions.
\end{proof}

Recall now a stability result about (pre)periodic points.

\begin{lem}[\cite{DH84}, cf. Lemma B.1 \cite{GM93}]
 \label{l:rep}
Let $g$ be a polynomial of degree $>1$, and $z$ be a repelling periodic point of $g$.
If an external ray $R_g(\theta)$ with rational argument $\theta$ lands at $z$, then,
 for every polynomial $\tilde g$ sufficiently close to $g$,
 the ray $R_{\tilde g}(\theta)$ lands at a repelling periodic point $\tilde z$ of $\tilde g$ close to $z$.
 Also, $\tilde z$ depends holomorphically on $\tilde g$ and has the same period as $z$.
\end{lem}

Let $A\subset \ol\C$ be any subset and $\Upsilon$ be a metric space with a marked base point $\tau_0$.
A map $(\tau,z)\mapsto \iota_\tau(z)$ from $\Upsilon\times A$ to $\ol\C$ is an \emph{equicontinuous motion
 (of $A$ over $\Upsilon$)} if $\iota_{\tau_0}=id_A$, the family of maps $\tau\mapsto \iota_\tau(z)$ parameterized by $z\in A$ is equicontinuous,  and $\iota_\tau$ is injective for every $\tau\in\Upsilon$.
An equicontinuous motion is \emph{holomorphic} if $\Upsilon$ is a Riemann surface,
 and each function $\tau\mapsto \iota_\tau(z)$, where $z\in A$, is holomorphic.
By the \emph{$\la$-lemma} of \cite{MSS}, to define a holomorphic motion, it is enough to require that
 every map $\iota_\tau$ is injective, and $\iota_\tau(z)$ depends holomorphically on $\tau$, for every fixed $z$.
Then the family of maps $\iota_\tau$ is automatically equicontinuous.
Suppose now that $F_\tau:\ol\C\to\ol\C$ is a family of rational maps such that $F_{\tau_0}(A)\subset A$.
An equicontinuous motion $(\tau,z)\mapsto \iota_\tau(z)$ is \emph{equivariant} with respect to the family $F_\tau$ if $\iota_\tau(F_{\tau_0}(z))=F_\tau(\iota_\tau(z))$ for all $z\in A$.
An $F_{\tau_0}$-invariant set $A$ is called \emph{stable} if $A$ admits an equivariant
 (with respect to the family $F_\tau$) holomorphic motion over some neighborhood of $\tau_0$ in $\Upsilon$.

\subsection{Cubic case}\label{ss:cubiq}
First we state results from \cite{BOPT16a} related to Subsection
\ref{ss:stab}; set $\Upsilon=\Fc_\la$ (see Subsection \ref{ss:slices1})
and $F_P=P$ for $P\in\Upsilon$ .

\begin{thm}
[Summary of some results of \cite{BOPT16a}]\label{t:holmot}
  Any $P\in\Fc_\la\sm\thu(\Pc_\la)$ is immediately renormalizable.
Its QL Julia set $J^*=\d K^*$ admits an equivariant holomorphic motion over
$\Fc_\la\sm\thu(\Pc_\la)$.
\end{thm}

The first statement of this theorem follows from Theorem C, and the second
statement follows from Theorem A and Lemma 3.12 of \cite{BOPT16a}.

A map $f\in\Fc_\la$ is \emph{stable} if $J_f=\d K_f$ is stable.
This notion is a
special case of $J$-stability \cite{lyu83,MSS}. Following \cite{Mc00}, we say
that a simple critical point $c_f$ of $f\in\Fc_\lambda$ is \emph{active} if,
for every small neighborhood $\Uc$ of $f$ in $\Fc_\lambda$, the sequence of
the mappings $g\mapsto g^{\circ n}(c_g)$ fails to be normal in $\Uc$. Here
$c_g$ is the critical point of $g$ close to $c_f$. If the critical point
$c_f$ is not active, then it is \emph{passive}. The map $f$ with simple
critical points is stable if and only if both critical points of $f$ are
passive \cite{Mc00}.

The set of all stable maps in $\Fc_\la$ is open; its components are called
\emph{stable components}. A classification of stable components of $\Fc_\la$
is given in Section 3 of \cite{Za}: a stable component $\Uc$ can be
\emph{hyperbolic-like}, \emph{capture}, or \emph{queer}.
If $\Uc$ is hyperbolic-like, then any $f\in\Uc$ has an attracting or super-attracting
 cycle whose immediate basin contains $\om_2$.
If $\Uc$ is capture, then
 $\om_2(f)$ is eventually mapped to a Fatou component $V$ containing $0$ or
 being an immediate parabolic basin of $0$.
Finally, if $\Uc$ is queer, then, for every $f\in\Uc$,
 the Julia set $J_f$ has positive measure
 and carries an $f$-invariant measurable line field.
The critical points of $f\in\Uc$ can be consistently denoted by $\om_1(f)$, $\om_2(f)$ so that
 $\om_2(f)\in J_f$, and $\om_1(f)$ is
 either in a (super)attracting/parabolic basin associated with $0$, or in $J_f$.
Moreover, the orbit of $\om_1(f)$ accumulates either on $0$ or on the boundary of the Siegel disk around $0$.
Conjecturally, there are no queer components.

\subsection{The structure of slices}\label{ss:slices}

Let us overview some results of \cite{slices}. Write
$\Cc_\la$ for the \emph{connectedness locus} in $\Fc_\la$, i.e., the set of
all $f\in\Fc_\la$ with $K_f$ connected. If $f\in\Fc_\la$ has disconnected
$K_f$, then the B\"ottcher coordinate $\phi_f$ extends to a disk containing
$\om^*_2(f)$. The latter is the so-called \emph{co-critical point} of $f$,
the unique point different from $\om_2$ and mapping to $f(\om_2)$. The map
$\Phi_\la(f)=\phi_f(\omega^*_2(f))$ is a conformal isomorphism between
$\Fc_\la\sm\Cc_\la$
  and the complement of the closed unit disk, cf. \cite{BH}.
Define \emph{parameter rays} $\Rc_\lambda(\theta)$ as the $\Phi_\la$-preimages of $\{re^{2\pi i\theta}\ |\ r>1\}$.
There is an explicit \cite{slices} set $\mathfrak{Q}$ of angle pairs $\theta_1$, $\theta_2\in\R/\Z$ such that
 the parameter rays $\Rc_\la(\theta_1)$, $\Rc_\la(\ta_2)$ land at the same point of $\Pc_\la$.
It is essential that the set $\mathfrak{Q}$ does not depend on $\la$ provided that $|\la|\le 1$.
For $\{\ta_1,\ta_2\}\in\mathfrak{Q}$, the domain $\Wc_\la(\ta_1,\ta_2)$ bounded by rays $\Rc_\la(\ta_1)$, $\Rc_\la(\ta_2)$
 and their common landing point so that $\Wc_\la(\ta_1,\ta_2)\cap\cuc_\la=\0$ is called a \emph{(parameter) wake}.
The common landing point of $\Rc_\la(\ta_1)$ and $\Rc_\la(\ta_2)$ is called
the \emph{root point} of $\Wc_\la(\ta_1,\ta_2)$. Observe that the terminology
used for cuts in the dynamic plane is different: a component of the
complement to a cut is called a \emph{wedge} and the common landing point of
the two external rays that form the cut is said to be the \emph{vertex} of
that cut. \emph{Limbs} are defined as intersections of $\Cc_\la$ with
parameter wakes. The following theorem summarizes the main results of
\cite{slices}.

\begin{thm}
  \label{t:slices}
  The connectedness locus $\Cc_\la$ is the disjoint union of $\cuc_\la$ and all limbs.
The set $\cuc_\la$ is a full continuum.
For every wake $\Wc_\la(\ta_1,\ta_2)$ and every $f\in\Wc_\la(\ta_1,\ta_2)$,
 the dynamic rays $R_f(\ta_1+1/3)$ and $R_f(\ta_2+2/3)$ lie in the same periodic cut attached to $K^*(f)$.
The vertex of this cut is either repelling or parabolic; in the latter case
it coincides with $0$.
\end{thm}

\smallskip

\noindent\textbf{Acknowledgments.} The authors are indebted to Micha\l{}
Misiurewicz who pointed out the example shown in Fig. \ref{fig:mis}.

\end{document}